\documentclass[reqno, 12pt]{amsart}

\usepackage{array}
\usepackage{amsmath}
\usepackage{amsfonts}
\usepackage{amssymb}
\usepackage{amsthm}
\usepackage{amsmath, amscd}
\usepackage{xy}
\xyoption{all}
\usepackage{marvosym}
\usepackage{hyperref}
	
\newtheorem{introtheorem}{Theorem}

\newtheorem{theorem}{Theorem}[section]
\newtheorem{Theorem}{Theorem}
\newtheorem{lemma}[theorem]{Lemma}
\newtheorem{proposition}[theorem]{Proposition}
\newtheorem{corollary}[theorem]{Corollary}

\newtheorem{definition}[theorem]{Definition}
\newtheorem{remark}[theorem]{Remark}

\newtheorem{conjecture}[Theorem]{Conjecture}

\def\N{\mathbb{N}}
\def\Z{\mathbb{Z}}
\def\C{\mathbb{C}}

\def\Q{\mathbb{Q}}
\def\F{\mathcal{F}}

\def\O{\mathcal{O}}
\def\A{\mathcal{A}}
\def\P{\mathbb{P}}
\def\der{\text{D}^{\text{b}}_{\text{coh}}(X)}
\def\dery{\text{D}^{\text{b}}_{\text{coh}}(Y)}
\def\derz{\text{D}^{\text{b}}_{\text{coh}}(Z)}
\def\ider{\emph{D}^{\emph{b}}_{\emph{coh}}(X)}
\def\idery{\emph{D}^{\emph{b}}_{\emph{coh}}(Y)}
\def\iderz{\emph{D}^{\emph{b}}_{\emph{coh}}(Z)}
\def\hom{\text{Hom}}

\def\tif{\text{if } }

\def\A{\mathcal{A}}
\def\L{\mathcal{L}}

\def\Fam{\{\mathcal{A}_i\}}

\title{
Reconstruction and finiteness results for Fourier-Mukai partners
}

\author{David Favero}

\address{Universit\"at von Wien
\newline\indent Fakult\"at f\"ur Mathematik 
\newline\indent Nordbergstr. 15 UZA4 C402
\newline\indent A-1090 Wien
\newline\indent Austria, Europe
\newline\indent{\bf favero@gmail.com}}

\numberwithin{equation}{section}

\begin{document}

\begin{abstract}
We develop some methods for studying the Fourier-Mukai partners of an algebraic variety.  As applications we prove that abelian varieties have finitely many Fourier-Mukai partners and that they are uniquely determined by their derived category of coherent $D$-modules.  We also generalize a famous theorem due to A. Bondal and D. Orlov.

\end{abstract}

\maketitle

\section{Introduction}
In \cite{Mu}, S. Mukai describes an equivalence between the derived category of sheaves on an abelian variety and and its dual.  This equivalence is achieved through a functor constructed in analogy with the Fourier transform.  The literature is now full of examples in which functors, mimicking the one provided by Mukai, give equivalences between the derived categories of two non-isomorphic varieties.  Such equivalences are called Fourier-Mukai transforms and two varieties which are equivalent by way of a Fourier-Mukai transform are called Fourier-Mukai partners.

In light of these equivalences, A. Bondal and D. Orlov pose a natural question in \cite{BO}: for a variety, how strong of an invariant is its derived category of sheaves?  Despite the aforementioned equivalences, they go on to prove that for a smooth projective variety with ample or anti-ample canonical bundle, the derived category is a complete invariant. More precisely, such a variety is uniquely determined from its bounded derived category of coherent sheaves.  This result is achieved through the rigidity that Serre duality imposes on the category.  Hence, it has been suggested that the derived category of sheaves on a variety becomes a stronger invariant as one passes further and further from the world of Calabi-Yau varieties.  Even so, the following conjecture postulates that the derived category of sheaves is a rather strong invariant on any smooth variety:

\begin{conjecture} [Kawamata \cite{Kaw}]\label{conj:1}
Let $X$ be a smooth (projective) variety.  There are finitely many isomorphism classes of varieties, $Y$, such that $X$ and $Y$ have equivalent bounded derived categories of coherent sheaves.
\end{conjecture}

\noindent This conjecture is known to hold for proper curves and surfaces \cite{H, Kaw}.  In this paper we will prove that this conjecture holds for abelian varieties over $\C$ and for varieties with ample or anti-ample canonical bundle (not necessarily proper).  These are special cases of the following two theorems:

\begin{introtheorem} \label{intro:abelian}
Let $X$ be a smooth projective variety over $\C$ and $\rho_X$ be the representation of the autoequivalence group of the derived category, $\emph{Aut}(\ider)$, on the de Rham cohomology group $\emph{H}^*(X, \Q)$.  If the kernel of $\rho_X$ is minimal in the sense that it is equal to $2\Z \times \emph{Aut}^0(X) \ltimes \emph{Pic}^0(X)$, then the number of isomorphism classes of projective varieties, $Y$, such that $\ider$ is equivalent to $\idery$ is bounded by the number of conjugacy classes of maximal unipotent subgroups of the image of $\rho_X$.  In particular, if the image of $\rho_X$ is an arithmetic subgroup of a semisimple lie group, then this number is finite.  Hence, as these conditions are satisfied by abelian varieties \cite{H,O,GLO}, Conjecture~\ref{conj:1} holds (with regards to projective Fourier-Mukai partners) for abelian varieties over $\C$. 
\end{introtheorem}

\begin{introtheorem} \label{intro:BO}
Let $X$ be a smooth variety.  Suppose that for any proper subvariety, $Z \subseteq X$, the canonical bundle of $X$, restricted to $Z$, is either ample or anti-ample (this is allowed to vary among the subvarieties).  If $Y$ is a smooth variety and $\ider$ is equivalent to $\idery$, then $X$ is isomorphic to $Y$.  In particular, Conjecture~\ref{conj:1} holds for $X$.
\end{introtheorem}

Using results from \cite{GLO}, we explicitly calculate the bound from Theorem~\ref{intro:abelian} for abelian varieties with Neron-Severi group equal to $\Z$.  This bound turns out to be the number of inequivalent cusps for the action of $\Gamma_0(N)$ on the upper half plane (N is a certain invariant of the abelian variety).  When $N$ is squarefree, the number of inequivalent cusps is precisely the number of Fourier-Mukai partners which are also abelian varieties, hence the bound is achieved.  However, the number of inequivalent cusps is not achieved by abelian varieties when $N$ is not squarefree.  As it turns out, any Fourier-Mukai partner of an abelian variety is in fact an abelian variety.  This statement has been proven, independent of this work, by D. Huybrechts together with M. Nieper-Wisskirchen in \cite{HN} and by F. Rosay in an unpublished thesis \cite{Ros}.  The author believes that when $N$ is not squarefree, the bound may be achieved by some more exotic spaces (not by varieties), for example see the equivalences in \cite{P}.

The moduli space of integrable local systems on an abelian variety satisfies the conditions of Theorem~\ref{intro:BO} (in fact all proper subvarieties are zero-dimensional in this case).  It was pointed out to the author by D. Arinkin that due to the equivalence of categories between the derived category of coherent sheaves on this moduli space and the derived category of coherent $D$-modules on the dual abelian variety \cite{L, PR, Roth}, one can use Theorem~\ref{intro:BO} toc recover an abelian variety from its derived category of $D$-modules.  Hence as a corollary to Theorem~\ref{intro:BO} we have the following theorem which we attribute to Arinkin.

\begin{introtheorem}[Arinkin]
If two abelian varieties $A$ and $B$ have equivalent derived categories of coherent $D$-modules then $A$ is isomorphic to $B$.
\end{introtheorem}

\noindent It has been conjectured by Orlov that if any two smooth varieties, $X$ and $Y$, have equivalent derived categories of coherent $D$-modules then they are isomorphic.

Theorems~\ref{intro:abelian} and~\ref{intro:BO} stem from the following two results of independent interest:
\begin{introtheorem} \label{intro:reconstruction}
Let $X$ and $Y$ be divisorial varieties, $F: \ider \to \idery$ an equivalence, $\Fam$ an ample family of line bundles on $X$, and $\{ \mathcal M_i \}$ any collection of line bundles on $Y$.  If tensor multiplication by $\mathcal A_i$ conjugates via $F$ to tensor multiplication by $\mathcal M_i$ then $X$ is isomorphic to $Y$.
\end{introtheorem}

\begin{introtheorem} \label{intro:recoverproper}
Let $X$ be a $k$-variety. Then $\mathfrak{Prop}(X)$ is equivalent to the full subcategory of $\ider$ consisting of objects $A \in \ider$ with the property that $\emph{Hom}(A,B)$ is finite dimensional over $k$ for all $B \in \ider$.  
\end{introtheorem}

Theorem~\ref{intro:reconstruction} informs us that any divisorial variety $Y$ such that $\der$ is equivalent to $\dery$ is encoded in the autoequivalence group of $\der$.  In the case of abelian varieties, the literature on this autoequivalence group is extensive, see for example \cite{GLO,H,O,P}.  This deep understanding of the autoequivalence group allows us to prove Theorem~\ref{intro:abelian}.  The author hopes that Theorem~\ref{intro:abelian} and Theorem~\ref{intro:reconstruction} may find further use when more is understood about the autoequivalence group of the derived category of a general variety.

Theorem~\ref{intro:recoverproper} is also readily applicable in that it allows us to redirect our attention from an equivalence between two categories with infinite dimensional Hom sets, to  an equivalence between two categories with finite dimensional Hom sets and Serre functors.  This is what produces Theorem~\ref{intro:BO}.  In principal, this method should extend to other situations.  For example, a generalization of Y. Kawamata's work in \cite{Kaw}, should be possible as well.

\section{Notations, conventions, and preliminaries}
As a matter of convention, a variety always means an integral separated scheme of finite type over a field $k$.  	For a variety $X$, we denote by $\der$, the bounded derived category of coherent sheaves of $\O_X$-modules.  The following proposition is well known and will be used implicitly, (see Proposition 3.5 in \cite{H}):
\begin{proposition}
Let $X$ be a noetherian scheme.  The natural functor:
$$
\ider \to \emph{D}^{\emph{b}}(\emph{\textbf{Qcoh}}(X))
$$
defines an equivalence between $\ider$ and the the full triangulated subcategory of 

\noindent $\emph{D}^{\emph{b}}(\emph{\textbf{Qcoh}}(X))$ consisting of objects with coherent cohomology.
\end{proposition}
All left or right exact functors, such as tensor products, pullbacks, and pushforwards, acting on objects in the derived category, are taken to be derived functors unless otherwise stated.  Furthermore, let $A$ be an object of $\der$, then we use the notation $\mu_A$ to denote the left derived functor of tensor multiplication by $A$.  That is, for $B \in \der$, we have $\mu_A(B) = A \otimes B$.
For two varieties, $X$ and  $Y$, when we posit the exist of an equivalence (or were we to discuss any functor for that matter), $F: \der \to \dery$, we always assume that $F$ is an exact functor i.e. $F$ preserves the $k$-linear and triangulated structures.
In fact most functors we discuss are integral transforms or Fourier-Mukai transforms.  We recall their definition:
\begin{definition} Let $X$ and $Y$ be varieties with $\mathcal P \in \emph{D}^{\emph{b}}_{\emph{coh}}(X \times Y)$.  Denote the two projections by,
\begin{eqnarray*}
q: X \times Y \to X &\text{and} &p:X \times Y \to Y.
\end{eqnarray*}

\noindent The induced \emph{integral transform} is the functor,
\begin{eqnarray*}
\Phi_{\mathcal P}: \ider \to \idery& , &A \mapsto p_*(q^*A \otimes \mathcal P).
\end{eqnarray*}
The object $\mathcal P$ is called the \emph{kernel} of the transform $\Phi_{\mathcal P}$.  Furthermore, the integral transform $\Phi_{\mathcal P}$ is called a \emph{Fourier-Mukai transform} if it is an equivalence.

\end{definition}

Similarly, when $X$ and $Y$ are compact algebraic varieties over $\C$, we can define a cohomological Fourier-Mukai transform for rational cohomology (implicitly using G.A.G.A.).

\begin{definition} \label{cohomological integral transform}
Let $X$ and $Y$ be compact algebraic varieties over $\C$ and $\alpha \in \emph{H}^*(X \times Y, \Q)$.  The \emph{cohomological integral transform} is defined as:
\begin{eqnarray*}
\Phi^H_{\alpha}: \emph{H}^*(X, \Q) \to \emph{H}^*(Y, \Q) & , &\beta \mapsto p_*(\alpha \cdot q^*(\beta)).
\end{eqnarray*}
For $\mathcal P \in \emph{D}^{\emph{b}}_{\emph{coh}}(X \times Y)$, as shorthand, we set $$\Phi^H_{\mathcal P} := \Phi^H_{\emph{ch}(\mathcal P) \cdot \sqrt{\emph{td}(X \times Y)}}.$$
The term, $\emph{ch}(\mathcal P) \cdot \sqrt{\emph{td}(X \times Y)}$, is called the \emph{Mukai vector} of $\mathcal P$.
\end{definition}
\noindent Using the Grothendeick-Riemann-Roch Theorem, one obtains (see \cite{H}, Proposition 5.33):
\begin{proposition}
If $\mathcal P \in \emph{D}^{\emph{b}}_{\emph{coh}}(X \times Y)$ defines an equivalence then the induced cohomological integral transform, $\Phi^H_{\mathcal P}$, is an isomorphism of rational vector spaces (not necessarily respecting the grading). 
\end{proposition}
\noindent Therefore, the autoequivalence group $\text{Aut}(\der)$ acts on $\text{H}^*(X, \Q)$.  Hence one obtains a representation, $\rho_X: \text{Aut}(\der) \to \text{Gl}(\text{H}^*(X, \Q))$.

For the reader's convenience, we now also recall the notion of an ample family of line bundles (see \cite{SGA} 6.II 2.2.3).

\begin{definition}
Let $X$ be a quasi-compact, quasi-separated scheme and $\Fam$ be a family of invertible sheaves on $X$.  The collection, $\Fam$, is called an \emph{ample family of line bundles} if it satisfies the following equivalent conditions:
\renewcommand{\labelenumi}{ \emph{\alph{enumi})}}
\begin{enumerate}
\item The open sets $X_f$ for all $f \in \Gamma(X,\A_i^{\otimes n})$ with $i\in I$, $n>0$ form a basis for the Zariski topology on $X$.

\item There is a family of sections $f \in \Gamma(X,\A_i^{\otimes n})$ such that the $X_f$ form an affine basis for the Zariski topology on $X$.

\item There is a family of sections $f \in \Gamma(X,\A_i^{\otimes n})$ such that the $X_f$ form an affine cover of $X$.

\item For any quasi-coherent sheaf $\F$ and $i \in I$, $n>0$ let $\F_{i,n}$ denote the subsheaf of $\F \otimes \A_i^{\otimes n}$ generated by global sections.  Then $\F$ is the sum of the submodules $\F_{i,n} \otimes \A_i^{\otimes -n}$.

\item For any quasi-coherent sheaf  of ideals $\F$ and $i \in I$, $n>0$, $\F$ is the sum of the submodules $\F_{i,n} \otimes \A_i^{\otimes -n}$.

\item For any quasi-coherent sheaf $\F$ of finite type there exist integers $n_i, k_i >0$ such that $\F$ is a quotient of $\bigoplus \limits_{i \in I}\A_i^{\otimes -n_i} \otimes \O_X^{k_i}$.

\item For any quasi-coherent sheaf of ideals $\F$ of finite type there exist integers $n_i, k_i >0$ such that $\F$ is a quotient of $\bigoplus \limits_{i \in I}\A_i^{\otimes -n_i} \otimes \O_X^{k_i}$.

\end{enumerate}

\end{definition}

\noindent A scheme which admits an ample family of line bundles is called divisorial.  All smooth varieties are divisorial (see \cite{SGA} 6.II).  More generally any normal noetherian locally $\Q$-factorial scheme with affine diagonal is divisorial \cite{BS}.  When discussing ample families in the proceeding section, we fix the following notation: $X$ is a noetherian scheme,  $\Fam$ is a finite ample family of line bundles on $X$, and hence $\mu_{A_i}$ is the autoequivalence of $\der$ which corresponds to tensoring with the sheaf $\A_i$.  We also use multi-index notation so that $\A^d := \A_1^{\otimes d_1} \otimes \cdots \otimes \A_r^{\otimes d_r}$ for $d \in \N^r$.

\section{Reconstructions and equivalences}

Using \cite{BO} as a guideline, we begin by proving Theorem~\ref{intro:reconstruction} from the introduction: the bounded derived category of coherent sheaves on a scheme together with the tensor multiplication functors corresponding to an ample family of line bundles uniquely determines the scheme.  The first step in our process, as in Bondal and Orlov's, is to recover the set of closed points - mildly disguised in the derived category as the isomorphism classes of structure sheaves of closed points concentrated in a fixed degree.  The following definition is a direct generalization of the one found in their paper:

\begin{definition}
  An object $P \in \ider$ is called a \emph{point object with respect to a collection of autoequivalences} $\{A_i\}$ if the following hold:
\renewcommand{\labelenumi}{\emph{\roman{enumi})}}
\begin{enumerate}
\item $A_i (P) \cong P$ for all $i$, 

\item $\emph{Hom}^{<0}(P,P) = 0$,

\item $\emph{Hom}^{0}(P,P) = k(P)$ with $k(P)$ a field.

\end{enumerate}
\end{definition}

\begin{proposition} \label{prop:recoverpoints}
Let $\Fam$ be an ample family of line bundles on a noetherian scheme $X$.  Then $P$ is a point object with respect to $\{\mu_{\mathcal A_i}\}$ if and only if $P \cong \O_x[r]$ for some $r \in \Z$ and some closed point $x \in X$.
\end{proposition}

\begin{proof}
Any shifted structure sheaf of a closed point is clearly a point object.  On the other hand suppose $P$ is a point object for $\{\mu_{\mathcal A_i}\}$.  
For each $f \in \Gamma(X, \A^d)$ we have an exact triangle $\O_X \to \A^d \to C$.  Notice that $C$ is a sheaf supported on $Z(f)$, the zero locus of $f$. Tensoring with $P$ we get a new triangle,
$P \to P \otimes \A^d \to P \otimes C$.  Since $P$ is a point object with respect to $\{\mu_{A_i}\}$, by  property (i), one has an isomorphism $P \cong P \otimes \A^d$.  Hence, by property (iii) the first morphism in the triangle, $\mu_P(f)$, is either zero or an isomorphism for any $f$.  If $\mu_P(f)$ is an isomorphism then $P \otimes C = 0$ hence, $f$ does not vanish on the support of $P$ i.e. $\text{Supp}(P) \cap Z(f) = \emptyset$.  If $\mu_P(f) = 0$ then $P$ is a summand of $P \otimes C$, hence $\text{Supp}(P) \subseteq Z(f)$.  In summary, for all $f$ one of the following is true; either $\text{Supp}(P) \cap Z(f) = \emptyset$ or $\text{Supp}(P) \subseteq Z(f)$.

Now suppose that $x,y$ are two distinct closed points in the support of $P$.   For each $f \in \Gamma(X, \A^d)$ and $d \geq 0$, denote by $X_f$ the open complement of $Z(f)$. Since $\Fam$ is an ample family, the collection of all $X_f$ form a basis for the Zariski topology of $X$.  Thus for some $d$ there exists a function $f \in \Gamma(X, \A^d)$ such that $f$ vanishes on $x$ but not on $y$, yielding a contradiction.  It follows that $P$ is supported at a point.  The result then follows from \cite{H} Lemma 4.5 (this lemma uses the noetherian assumption).
\end{proof}

We have shown that one can recover the structure sheaves of points up to shift.  Continuing with Bondal and Orlov's work as a guideline,  we now wish to recover the line bundles up to shift.  This motivates the following definition which is once again a direct generalization of their work:
\begin{definition}
  An object $L \in \ider$ is called \emph{invertible for a set S} if for all $P \in S$ there exists an $n_P \in \Z$ such that :

$$
\emph{Hom}(L,P[i]) = \begin{cases}
k(P) & \tif i=n_P \\
0 & otherwise 
\end{cases}
$$
We say an object is \emph{invertible with respect to a collection of autoequivalences} $\{A_i\}$ if it is invertible for the set of point objects with respect to $\{A_i\}$.
\end{definition}

\begin{proposition}\label{prop:recoverinvertibles}
Let $\Fam$ be an ample family of line bundles on a noetherian scheme $X$.  Then $L$ is an invertible object with respect to $\{\mu_{A_i}\}$ if and only if $L \cong \L[r]$ for some $r \in \Z$ and some line bundle $\L$.
\end{proposition}
\begin{proof}
The proof of Proposition 2.4 in \cite{BO} applies directly to this situation.
\end{proof}

\begin{lemma}\label{lem:pointstopoints}
Suppose $X$ and $Y$ are divisorial varieties and $F: \ider \to \idery$ is an equivalence which maps objects with zero dimensional support to objects with zero-dimensional support then $X$ is isomorphic to $Y$.  Similarly, if $F$ maps shifted skyscraper sheaves of closed points to shifted skyscraper sheaves of closed points then $X$ is isomorphic to $Y$.    
\end{lemma}


\begin{proof}
Let $R_X$ denote the full subcategory of $\der$ consisting of objects with zero dimensional support.  Consider objects $P \in R_X$ satisfying:
\renewcommand{\labelenumi}{\emph{\roman{enumi})}}
\begin{enumerate}
\item $\text{Hom}^{<0}(P,P) = 0$,

\item $\text{Hom}^{0}(P,P) = k(P)$ with $k(P)$ a field.

\end{enumerate}
Call this class of objects $T_X$.  By Lemma 4.5 of \cite{H}, all such objects are isomorphic to $\O_x[r]$ for some $r \in \Z$ and $x$ a closed point of $X$.

By Proposition~\ref{prop:recoverinvertibles}, invertible objects for $T_X$ are precisely the objects isomorphic to shifted line bundles.  Fix an invertible object $N$ and  let $X_0$ be the set of isomorphism classes of objects $P \in T_X$ such that $\hom (N,P) \neq 0$.  This set is in bijection to $\{\O_x[r]\}$, where $x$ is a closed point of  $X$ and $r \in \Z$ is fixed.  Thus, $X_0$ can be thought of as the set of closed points of $X$.

If $F: \der \to \dery$ is an equivalence which sends $R_X$ to $R_Y$, then it sends $T_X$ to $T_Y$ and it follows that it sends $X_0$ to $Y_0$, where $Y_0$ is constructed by fixing the invertible object $F(N)$.  Hence it induces a set theoretic map $f: X_0 \to Y_0$.  This map is clearly injective as $F$ is an equivalence.  To show that $f$ is surjective, observe that the set $Y_0$ has the property that for any object $B \in \dery$ there exists a $[P] \in Y_0$ such that $\text{Hom}(B, P[n]) \neq 0$ for some $n$.  Similarly, the set $X_0$ has the same property for any object of $\der$.  Hence, as $F$ is an equivalence, $f(X_0)$ has this property.  One sees easily that no proper subset of $Y_0$ has this property.     Therefore $f$ is surjective.

We proceed by recovering the Zariski topology on $X_0$ and $Y_0$.  By Proposition~\ref{prop:recoverinvertibles}, all objects isomorphic to line bundles (shifted by the same fixed $r$ as above) are characterized by,
$$
\text{Pic}_{X,N} := \{ L | L \text{ is invertible and }\text{Hom}(L, P) = k(P) \; \forall P \in X_0 \}.
$$
Now given any two objects $L_1, L_2 \in \text{Pic}_{X,N}$, a morphism $\alpha \in \hom(L_1, L_2)$, and an object $P \in X_0$,  we get an induced map,
$$
\alpha^*_P: \hom(L_2, P) \to \hom(L_1, P).
$$
Denote by $X_\alpha$ the subset of those objects $P \in X_0$ for which $\alpha^*_P \neq 0$.  Then $X_\alpha$ is the subset of $X_0$ which is the complement of the zero-locus of $\alpha$.  Since $X$ is divisorial, letting $\alpha$ run over all morphisms in $\hom (L_1, L_2)$ and $L_1, L_2$ run over all line bundles, we get a basis for the Zariski topology on $X_0$.  The same is done for $Y_0$ using $F(N)$.  As the equivalence induces a bijection from  $X_0$ to $Y_0$, it also induces a bijection from $\text{Pic}_{X,N}$ to $\text{Pic}_{Y,F(N)}$.  Hence, the map $f$ is a homeomorphism of topological spaces.

We now recover the structure sheaf on $X_0$ and $Y_0$.  For each open set $U \subseteq X$ there is a corresponding open set $U_0 \subseteq X_0$.  Define $\text{D}_{X \backslash U}$ to be the full subcategory of $\der$ consisting of objects $A$ such that $\text{Hom}(A,P[i])= 0$ for all $ P \in U_0$ and all $i \in \Z$.  This is the subcategory of objects supported on $X \backslash U$.  Localizing we reconstruct $\text{D}^{\text{b}}_{\text{coh}}(U)$ i.e. $\text{D}^{\text{b}}_{\text{coh}}(U)$ is the Verdier quotient of $\der$ by $D_{X \backslash U}$.  Hence we can reconstruct the structure sheaf on $X_0$ as $\O_{X_0}(U_0) := \text{Hom}_{D^b_{\text{coh}}(U)}(N,N)$, giving $X_0$ the structure of a locally ringed space.  Doing the same for $Y_0$, we surmise that $f$ is an 
isomorphism of locally ringed spaces.    As $X$ and $Y$ are varieties, they must also be isomorphic as schemes.

The first statement in the lemma is proven.  The later statement is easier - it starts with the data of $T_X$, which we recovered from $R_X$.



\end{proof}

\begin{remark}
When an equivalence like the one in the above lemma is given by a Fourier-Mukai transform, then there is a stronger result.  Namely, such a Fourier-Mukai transform is supported on the graph of an isomorphism between $X$ and $Y$ i.e. it is the composition of tensoring with a line bundle, an isomorphism, and a shift (see for example \cite{BO,H}).  Recent results, independently obtained by M. Ballard in \cite{Bal} and V. Lunts and D. Orlov in \cite{LO} show that any equivalence between quasi-projective varieties can be approximated by a Fourier-Mukai transform, meaning there is a Fourier-Mukai functor which acts the same way on objects.  Hence, from these works, the above lemma follows immediately for quasi-projective varieties.
\end{remark}

\begin{remark} The proof applies to a larger class of schemes than just varieties.  Instead suppose $X$ is a noetherian scheme and let $X_0$ be the set of closed points of $X$.   The proof requires that the prime ideals are in bijection with irreducible closed subsets of the set of closed points.  Many of the proofs below apply to this situation as well.
\end{remark}

We now arrive at our first theorem,
\begin{theorem} \label{thm:reconstruction}
Let $X$ and $Y$ be divisorial varieties, $F: \der \to \dery$ an equivalence, $\Fam$ an ample family of line bundles on $X$, and $\{ \mathcal M_i \}$ any collection of line bundles on $Y$.  If $F^{-1} \circ \mu_{\mathcal M_i} \circ F = \mu_{\mathcal A_i}$ then $X$ is isomorphic to $Y$.
\end{theorem}

\begin{proof}
Suppose $F: \der \to \dery$ is an equivalence and $F^{-1} \circ \mu_{\mathcal M_i} \circ F = \mu_{\mathcal A_i}$.  Then $F$ takes point objects with respect to $\{\mu_{\mathcal A_i}\}$ to point objects with respect to $\{\mu_{\mathcal M_i} \}$.  By Proposition~\ref{prop:recoverpoints}, all point objects with respect to $\{\mu_{\mathcal A_i}\}$ are isomorphic to $\O_x[r]$ for some closed point $x \in X$ and some $r \in \Z$.  Since for any closed point $y \in Y$ and any $n \in \Z$ the objects $\O_y[n]$ are point objects with respect to $\{\mu_{\mathcal M_i} \}$ we have that $F(\O_y[n])$ is isomorphic to an object of the form $\O_x[r]$.  The theorem follows from Lemma~\ref{lem:pointstopoints}.
\end{proof}

\begin{remark}
 To obtain this result, it is not strictly necessary to use the bounded derived categories of coherent sheaves.  For quasi-projective varieties it is known that the bounded derived category of coherent sheaves can be recovered from the derived category of quasi-coherent sheaves; it is precisely the full subcategory of locally cohomologically finitely presented objects(in fact the statement is true for a much larger class of schemes satisfying a certain technical condition, see \cite{Ro2} for details).  Hence one could begin by positing an equivalence between the derived categories of quasi-coherent sheaves on two quasi-projective varieties $X$ and $Y$.  One could also begin by positing an equivalence between the bounded derived category of quasi-coherent sheaves on two divisorial varieties $X$ and $Y$.  In this latter case, the bounded derived category of coherent sheaves is encoded as the full subcategory of compact objects (this works for any separated noetherian scheme, see \cite{Ro2}).
\end{remark}

\begin{remark}
Let $X$, $Y$, and $Z$ be divisorial varieties and $F: \idery \to \ider$, $G: \iderz \to \ider$ be equivalences.  Let $\{\A^Y_i\}$ be an ample family of line bundles  on $Y$ and $\{\A^Z_i\}$  be an ample family of line bundles on $Z$.  Then we have two collections of autoequivalences of $\ider$, namely $\{ F \circ \mu_{\A^Y_i} \circ F^{-1} \}$ and  $\{ G \circ \mu_{\A^Z_i} \circ G^{-1} \}$.  If these collections are isomorphic, then the theorem tells us that $Y$ is isomorphic to $Z$.  Therefore varieties with equivalent derived categories to $\ider$ are encoded in $\emph{Aut}(\ider)$ via collections as above.  This will be used in the next section to study Fourier-Mukai partners.  However, the author believes that more could be accomplished were one to identify categorical conditions which isolate collections of autoequivalences of this form.
\end{remark}

\begin{corollary}
Suppose $X$ is a quasi-affine variety and $Y$ is any divisorial variety.  Suppose there exists an equivalence, $F: \ider \to \idery$, then $X$ is isomorphic to $Y$.
\end{corollary}

\begin{proof}
The functor $\mu_{\O_X}$ is the identity functor, hence for any equivalence $F: \der \to \dery$ we have $F^{-1} \circ \mu_{\O_X}\circ F \cong \mu_{\O_Y}$.  Since $X$ is quasi-affine, $\{\O_X\}$ is an ample family.
\end{proof}


We would now like to extend Bondal and Orlov's aforementioned result to non-proper varieties (in fact the above corollary will be a special case of our generalization).  To do this, we begin by considering, $\mathfrak{Prop}(X)$, the full subcategory of $\der$ consisting of objects supported on proper subvarieties of $X$. 

\begin{theorem} \label{thm:recoverpropersupport}
Let $X$ be a $k$-variety. Then $\mathfrak{Prop}(X)$ is equivalent to the full subcategory of $\ider$ consisting of objects $A \in \ider$ with the property that $\emph{Hom}(A,B)$ is finite dimensional over $k$ for all $B \in \ider$.
\end{theorem}

\begin{proof}
Suppose $A$ is supported on a proper subscheme.  Consider the spectral sequence,
$$
E_2^{(p,q)} = \text{Hom}_{\der}(\text{H}^{-q}(A),B[p]) \Rightarrow \text{Hom}_{\der}(A,B[p+q]).
$$


Each term in the spectral sequence is a finite dimensional vector space hence $$\text{Hom}_{\der}(A, B)$$ is finite dimensional.

On the other hand, suppose $A$ is supported on a non-proper subscheme.  Let $m$ be the greatest integer such that $\text{H}^{m}(A)$ is supported on a non-proper subscheme.  Since the support is not proper, there exists an affine curve $C \subseteq \text{Supp}(\text{H}^m(A))$.  Let $i$ denote the inclusion map, $i: C \to \text{Supp}(\text{H}^m(A))$.  We compute $\text{Hom}_{\der}(A,i_*i^*\text{H}^m(A))$ (the pullback is not derived so that $i_*i^*\text{H}^m(A)$ is a sheaf concentrated in degree zero).  Using the same spectral sequence as before we have,
$$
E_2^{(p,q)} = \text{Hom}_{\der}(\text{H}^{-q}(A),i_*i^*\text{H}^m(A)[p]) \Rightarrow \text{Hom}_{\der}(A,i_*i^*\text{H}^m(A)[p+q]).
$$
Notice that since C is affine $E_2^{(0,-m)} \cong \text{Hom}(i^*\text{H}^m(A),i^*\text{H}^m(A)) $ is the endomorphism ring of a module supported on all of $C$.  Hence this is an infinite dimensional vector space.  Furthermore all the terms below it in the spectral sequence are finite dimensional and all the terms to the left are zero.    As the differentials point downward and to the right, it follows that this term is affected by finite dimensional vector spaces only a finite number of times before it stabilizes.  Hence $E_\infty^{(0,-m)}$ is infinite dimensional.  Thus there is a filtration of $$\text{Hom}_{\der}(A,i_*i^*\text{H}^m(A)[-m])$$ which contains an infinite dimensional vector space.  Hence $$\text{Hom}_{\der}(A,i_*i^*\text{H}^m(A)[-m])$$ is infinite dimensional.
\end{proof}

\begin{remark} This statement is formally similar to a result of Orlov's in \cite{O3}.  This result says that the category $\mathfrak {Perf}(X)$ formed by perfect complexes can be recovered from the (unbounded) derived category of coherent sheaves.  More precisely, the triangulated subcategory formed by perfect complexes is exactly the full subcategory of homologically finite objects.
\end{remark}

\begin{lemma}\label{lem:equaldimension}
Suppose $X$ and $Y$ are smooth varieties with equivalent bounded derived categories of coherent sheaves.  Then $\emph{dim }X = \emph{dim }Y$.
\end{lemma}

\begin{proof}
By Theorem~\ref{thm:recoverpropersupport}, $F$ restricts to an equivalence $F: \mathfrak{Prop}(X) \to \mathfrak{Prop}(Y)$.  
Let $\omega_X$ be the canonical bundle of $X$.  The category, $\mathfrak{Prop}(X)$, comes equipped with a Serre functor $S_X \cong \mu_{ \omega_{X}}[\text{dim }X]$.  By uniqueness of the Serre functor (see \cite{H}, Lemma 1.30), $F \circ S_X \circ F^{-1} \cong S_Y$.  Hence point objects with respect to the $S_X[-\text{dim} X]$ are sent by $F$ to point objects with respect to $S_Y[-\text{dim} X]$ and vice versa.  However, one easily verifies that there exist non-zero point objects with respect to $S_Y[r]$ if and only if $r = -\text{dim }Y$.  Hence $\text{dim }X = \text{dim }Y$.
\end{proof}

\begin{theorem} \label{thm:BOgeneralization}
Let $X$ be a smooth variety.  Suppose that for any proper subvariety, $Z \subseteq X$, the canonical bundle of $X$, restricted to $Z$, is either ample or anti-ample (this is allowed to vary among the subvarieties).  If $Y$ is a smooth variety and $\ider$ is equivalent to $\idery$, then $X$ is isomorphic to $Y$.  
\end{theorem}

\begin{proof}
The proof begins exactly the same way as in the lemma above.  
That is, by Theorem~\ref{thm:recoverpropersupport}, $F$ restricts to an equivalence $F: \mathfrak{Prop}(X) \to \mathfrak{Prop}(Y)$, which preserves Serre functors.
Now, clearly for any closed point $y \in Y$ and any $s\in \Z$, the object $\O_y[s]$ is a point object for $S_Y[-\text{dim }Y]$.  Hence $F^{-1}(\O_y)[s]$ is a point object with respect to $S_X[-\text{dim } Y]$.  By the above lemma, $S_X[-\text{dim }Y] = S_X[-\text{dim }X]$.  Now $F^{-1}(\O_y)[s]$ is properly supported, hence by assumption $S_X[-\text{dim }X]$ acts on $F^{-1}(\O_y)[s]$ by tensoring with an ample or anti-ample line bundle.  If follows from the proof of Proposition~\ref{prop:recoverpoints}, that since $F^{-1}(\O_y)[s]$ is a point object with respect to $S_X[-\text{dim }X]$, it is isomorphic to $\O_x[r]$ for some closed point $x \in X$ and some $r \in \Z$.  The theorem follows from Lemma~\ref{lem:pointstopoints}.
\end{proof}

As mentioned in the introduction, it was pointed out to the author by D. Arinkin that the moduli space of integrable systems on an abelian variety satisfies the conditions of the above theorem.  This leads to the following:

\begin{theorem}[Arinkin]
If two abelian varieties $A$ and $B$ have equivalent derived categories of coherent $D$-modules then $A \cong B$ 
\end{theorem}

\begin{proof}
Let $\mathfrak g = \text{H}^1(A, \O_A)$.  Then there is a tautological extension
$$
0 \to \mathfrak g^* \otimes \O_A \to \mathcal E \to \O_A \to 0
$$  
which corresponds to the identity of $\text{End}(\mathfrak g^*) = \text{Ext}^1(\O_A, \mathfrak g^* \otimes \O_A)$.  Let $A^\natural$ be the $\mathfrak g^*$-principal bundle associated to the extension $\mathcal E$.  Then the derived category of $D$-modules on $A$ is equivalent to the category of coherent sheaves on $\hat{A}^\natural$ where $\hat{A}$ is the dual abelian variety \cite{L,Roth, PR}.  Hence, by assumption, there is an equivalence $F: \text{D}^\text{b}_{\text{coh}}(\hat{A}^\natural) \to \text{D}^\text{b}_{\text{coh}}(\hat{B}^\natural)$

We now show that for an abelian variety $A$, the space $A^\natural$ has only finite sets of points as proper subvarieties.  Suppose $P \subseteq A^\natural$ is a proper subvariety.  Since $A^\natural$ is an affine bundle over $A$, the projection $\pi: P \to A$ is finite and the pullback of $A^\natural$ to $P$ will have a section i.e. it will be a trivial affine bundle.  Now $A^\natural$ is represented by the ample class $ \text{Id} \in \text{End}(\mathfrak g^*) = \text{Ext}^1(\O_A, \mathfrak g^* \otimes \O_A)= \text{H}^{1}(A,\Omega^{1}_{A})$.   Now since $\pi: P \to A$ is finite, the projection of $\pi^*(\text{Id}) \in \text{H}^{1}(P, \pi^* \Omega^{1}_{A})$ onto $\text{H}^{1}(P,\Omega^{1}_{P})$ is also ample.  The only way an ample class on $P$ can be zero is if it is a finite set of points.

We may now apply Theorem~\ref{thm:BOgeneralization} to the equivalence $F: \text{D}^\text{b}_{\text{coh}}(\hat{A}^\natural) \to \text{D}^\text{b}_{\text{coh}}(\hat{B}^\natural)$ to yield an isomorphism between $\hat{A}^\natural$ and $\hat{B}^\natural$.  Using Hodge theory one deduces that $\hat{A}$ is isomorphic to $\hat{B}$.  Dualizing this isomorphism yields the result.
\end{proof}

\begin{remark} Actually over $\C$, $A^\natural$ is a Stein space and hence we see immediately that finite sets of points are the only proper closed subvarieties.
\end{remark}

\section{Autoequivalences and Fourier-Mukai partners}
In this section, we limit our focus to the case of smooth projective varieties.  This has a number of advantages.  First, for projective varieties, a single ample line bundle gives an ample family.   Second, due to a famous result of Orlov \cite{O2}, generalized by A. Canonaco and P. Stellari \cite{CP}, any equivalence between derived categories of smooth projective varieties is a Fourier-Mukai transform (in what follows, this fact is often used implicitly without further mention).  Third, we may use the  cohomological integral transform introduced above (see Definition~\ref{cohomological integral transform}.  However, due to previously mentioned work of M. Ballard in \cite{Bal} and V. Lunts and D. Orlov in \cite{LO}, much of what lies below can be extended to the quasi-projective situation.

Now suppose that we have two projective varieties $X$ and $Y$ and an equivalence $F: \der \to \dery$.  This equivalence induces an isomorphism between the groups of autoequivalences which we denote by, 
\begin{eqnarray*}
F_*: \text{Aut} (\der)  \to  \text{Aut}(\dery) &,&\Phi  \mapsto F \circ \Phi \circ F^{-1}.
\end{eqnarray*}
Suppose $\A$ is an ample line bundle on $X$.  As a consequence of Theorem~\ref{thm:reconstruction}, if  $F_*(\mu_{\mathcal A})= \mu_{\mathcal M}$ for some line bundle $\mathcal M$ on $Y$, then $X$ is isomorphic to $Y$.  In fact since $F$ is a Fourier-Mukai transform we can say more and indeed the following generalization will be useful in our applications.

\begin{lemma} \label{lem:twist}
Let $X$ and $Y$ be smooth projective varieties.  Let $\mathcal A$ be an ample line bundle on $X$, $\tau \in \emph{Aut}(Y)$, $\mathcal L \in \emph{Pic }Y$, and $r \in \Z$.  Suppose we have an equivalence $F: \ider \cong \idery$ such that $F_*(\mu_{\mathcal A}) \cong \mu_{ \mathcal L}\circ \tau_*[r]$.  Then $F$ is isomorphic as a functor to $\mu_{\mathcal N} \circ \gamma_* [s]$ for some line bundle $\mathcal N \in \emph{Pic}(Y)$ , some isomorphism $\gamma: X \tilde{\longrightarrow} Y$, and some $s \in \Z$.
\end{lemma}



\begin{proof}
First we prove $r=0$.  Pick any closed point $x \in X$ and notice that  $\O_x \cong \O_x \otimes \A = \mu_{\mathcal A}(\O_x)$.  Hence,
\begin{align*}
F(\O_x) &\cong (F \circ \mu_{\mathcal A})(\O_x) \\
&\cong (\mu_{ \mathcal L}\circ \tau_*[r] \circ F)(\O_x) \\
&\cong (\mu_{ \mathcal L}\circ \tau_*) \circ F(\O_x) [r].
\end{align*}
Comparing the cohomology sheaves on the left hand side with those on the right hand side, one concludes that $r=0$.  In other words one has $F_*(\mu_{\mathcal A}) \cong \mu_{ \mathcal L}\circ \tau_*$.


Next we show that there exists an $n$ such that $\tau^n = \text{id}$.  We have $F \cong \Phi_P$ and $F^{-1} \cong \Phi_Q$ for some $P,Q \in D^b(X \times Y)$.  Consider the object $P \boxtimes Q \in D^b(X \times X \times Y \times Y)$.  This object also defines an equivalence \cite{H},
$$
\Phi_{P \boxtimes Q}: D^b(X \times X) \to D^b(Y \times Y).
$$
Suppose $\Phi_S$ and $\Phi_T$ are autoequivalences such that $F_*(\Phi_S) \cong \Phi_T$.  A simple calculation reveals that, 
$$\Phi_{P \boxtimes Q} (S) \cong T.$$
Now let $\Delta$ be the diagonal map and $\delta^\tau_n := \tau_*\L \otimes \tau_*^2 \L \otimes ... \otimes \tau_*^n \L$.  One easily verifies that the following equivalences are paired with the corresponding kernels below,
\begin{eqnarray*}
\underbrace{\mu_{\mathcal A}\circ \cdots \circ \mu_{\mathcal A} }_{n\text{ times}}& \cong& \Phi_{\Delta_*(\mathcal A^{\otimes n})}, \\
\underbrace{\mu_{ \mathcal L} \circ \tau_*\circ \cdots \circ \mu_{ \mathcal L} \circ \tau_* }_{n\text{ times}}& \cong &\Phi_{(\text{id} \times \tau^n)_*\delta^\tau_n} .
\end{eqnarray*}
Putting all this together we have, 
$$
\Phi_{P \boxtimes Q} (\Delta_*(\mathcal A^{\otimes n})) \cong (\text{id} \times \tau^n)_*\delta^\tau_n.$$

Now let $Z_n$ denote the fixed locus of $\tau^n$.  Then we have,
\begin{align*}
\text{Hom}_X(\O_X , \mathcal A^{\otimes n}) &\cong \text{Hom}_{X \times X}(\Delta_* \O_X, \Delta_* \mathcal A^{\otimes n}) \\
& \cong \text{Hom}_{Y \times Y}(\Phi_{P \boxtimes Q}(\Delta_* \O_X), \Phi_{P \boxtimes Q}(\Delta_* \mathcal A^{\otimes n})) \\
&\cong \text{Hom}_{Y \times Y}(\Delta_*\O_Y, (\text{id} \times \tau^n)_*\delta^\tau_n ) \\
& \cong \text{Hom}_Y(\O_Y, \Delta^!(\text{id} \times \tau^n)_*\delta^\tau_n ). 
\end{align*}
The final term is the zeroth cohomology of a complex supported on $Z_n$.  Therefore, since $\mathcal A$ is ample, there exists an $n$ such that $Z_n \neq \emptyset$.  Now, for $z \in Z_n$, $\O_z$ is a point object with respect to 
$$
\tau^n_* \circ \mu_{\delta^\tau_n} \cong \underbrace{\mu_{ \mathcal L} \circ \tau_*\circ \cdots \circ \mu_{ \mathcal L} \circ \tau_*}_{n\text{ times}} .
$$  Hence $\Phi_Q(\O_z)$ is a point object with respect to $\mu_{\mathcal A^{\otimes n}}$.  From Proposition~\ref{prop:recoverpoints} one obtains an isomorphism   $\Phi_Q(\O_z) \cong \O_x[r]$ for some closed point $x \in X$ and some $s  \in \Z$.

As $\Phi_Q(\O_z) \cong \O_x[r]$, Corollary 6.12 of \cite{H} posits the existence of an open set $U$ (of the set of closed points) and a morphism $f: U \to X$ such that for any $y \in U$ we have that $\Phi_Q (\O_y) \cong \O_{f(y)}[s]$.  Moreover, $f$ must be injective since F is an equivalence.  Let $v \in f(U)$, then as $\O_v$ is a point object with respect to $\mu_{{\mathcal A}^{\otimes n}}$, it follows that $\O_{f^{-1}(v)}$ is a point object with respect to $\tau^n_* \circ \mu_{\delta^\tau_n}$.  Therefore $f^{-1}(v)$ is a fixed point of $\tau^n$.  Therefore the fixed locus contains $U$, but as the fixed locus is closed and $X$ is irreducible, the fixed locus is the whole space, i.e. $\tau^n = \text{id}$.

Now, as $F_*(\mu_{\mathcal A}) \cong \mu_{ \mathcal L}\circ \tau_*$ , one has
\begin{align*}
F_*(\mu_{\mathcal A^{\otimes n}}) & \cong \underbrace{\mu_{ \mathcal L} \circ \tau_*\circ \cdots \circ \mu_{ \mathcal L} \circ \tau_*}_{n\text{ times}} \\
& \cong \tau^n_* \circ \mu_{\delta^\tau_n}  \\
& \cong \mu_{\delta^\tau_n} .
\end{align*}
As in the proof of Theorem~\ref{thm:reconstruction}, it follows that for any closed point $y \in Y$,  one has $F^{-1}(O_y) \cong O_x[t]$ for some closed point $x \in X$ and some $t \in \Z$.  At this stage, one could use Lemma~\ref{lem:pointstopoints} to prove that $X$ is isomorphic to $Y$ but in this context we can say even more.  As $F$ is a Fourier-Mukai transform, Corollary 5.23 in \cite{H} tells us that $F$ is isomorphic as a functor to $\mu_{\mathcal N} \circ \gamma_* [s]$ for some line bundle $\mathcal N \in \text{Pic}(Y)$ , some isomorphism $\gamma: X \tilde{\longrightarrow} Y$, and some $s \in \Z$.

\end{proof}

\begin{remark}  
If $\mathcal A$ is not ample then the above is not necessarily true.  For example, take $\mathcal P \in D^b_{\text{coh}}(A \times \hat{A})$ to be the Poincare line bundle on an abelian variety $A$.  Let $\L$ be a degree zero line bundle on $A$ thought of as an element of $\hat A$. Further, let $t_{\L}$ denote translation by $\L \in \hat{A}$, then ${\Phi_{\mathcal P}}_*(\mu_{\L}) \cong {t_{\L}}_*$.  However, $\Phi_{\mathcal P}$ is not isomorphic to a functor which is the composition of tensoring with a line bundle and pushing forward by an isomorphism.
\end{remark}

To illustrate how Lemma~\ref{lem:twist} can be utilized to bound the number of Fourier-Mukai partners of a given variety we provide some easy corollaries here.  These will serve as a warm-up to Theorem~\ref{thm:representation}.

\begin{corollary} \label{cor:maximalabelian}
The number of projective Fourier-Mukai partners of a smooth projective variety, $X$, is bounded by the number of conjugacy classes of maximal abelian subgroups of $\emph{Aut}(\ider)$.
\end{corollary}

\begin{proof}
Notice that for a variety $X$, the Picard group, $\text{Pic}(X)$, acts by tensor multiplication on $\der$ - we use this to identify $\text{Pic}(X)$ with a subgroup of $\text{Aut}(\der)$.  Suppose $Y$ and $Z$ are projective Fourier-Mukai partners with equivalences given by $F: \der \to \dery$ and $G: \der \to \derz$.  In addition, assume that ${F^{-1}}_*(\text{Pic}(Y))$ and ${G^{-1}}_*(\text{Pic}(Z))$ are contained in conjugate maximal abelian subgroups of $\text{Aut}(\der)$.  Hence there exists $\psi \in \text{Aut}(\der)$ such that $(\psi \circ {F^{-1}})_*(\text{Pic}(Y))$ is contained in the same maximal abelian subgroup of $\text{Aut}(\der)$ as ${G^{-1}}_*(\text{Pic}(Z))$.  Consider the equivalence $G \circ \psi \circ F^{-1}: \dery \to \derz$.
It follows that $(G \circ \psi \circ F^{-1})_* (\text{Pic}(Y))$ commutes with $\text{Pic}(Z)$.  In particular, this subgroup commutes with $\mu_{\mathcal B}$, where $\mathcal B$ is an ample line bundle on $Z$.  From the above lemma,  any element of $(G \circ \psi \circ F^{-1})_* (\text{Pic}(Y))$ is of the form $\mu_{\mathcal L} \circ \gamma_* [s]$ for some line bundle $\mathcal N \in \text{Pic}(Z)$, some isomorphism $\gamma: Y \tilde{\longrightarrow} Z$, and some $s \in \Z$.  Let $\mathcal A$ be an ample line bundle on $Y$.   In particular, one has that $(G \circ \psi \circ F^{-1})_*(\mu_{\mathcal A}) \cong \mu_{\mathcal L} \circ \gamma [s]$ for some line bundle $\mathcal N \in \text{Pic}(Z),$ some isomorphism $\gamma: Y \tilde{\longrightarrow} Z$, and some $s \in \Z.$  Applying the above lemma again, one discovers that $Y$ and $Z$ are isomorphic.
\end{proof}

Using similar reasoning, one could also say,

\begin{corollary}
Suppose $X$ is a smooth projective variety such that for every autoequivalence of $\ider$, $\psi \in \emph{Aut}(\ider)$, there exists a power of $\psi$ which is conjugate to $\mu_{\mathcal N} \circ \gamma_*[s]$ for some $\gamma \in \emph{Aut}(X)$, $\mathcal N \in \emph{Pic}(X)$, and $s \in \Z$.  Then $X$ has no non-trivial Fourier-Mukai partners.
\end{corollary}

\begin{proof}
Suppose $Y$ is a smooth projective variety and $F: \dery \cong \der$.  Let $\mathcal A$ be an ample line bundle on $Y$.  Then by hypothesis there exists an $n \in \N$ and an  $\alpha \in \text{Aut}(\der)$ such that $F_* (\mu_{ \mathcal A^{\otimes n}})  = \alpha_* (\mu_{\mathcal N} \circ \gamma_*[s])$.  So we have $(\alpha^{-1} \circ F)_*(\mu_{\mathcal A^{\otimes n}}) = \mu_{\mathcal N} \circ \gamma_*[s].$  The lemma implies that $X$ is isomorphic to $Y$.
\end{proof}


The ideas of the above corollaries lead us to our main result,
\begin{theorem}  \label{thm:representation}
Let $X$ be a smooth projective variety over $\C$ and $\rho_X$ be the representation of the autoequivalence group of the derived category, $\emph{Aut}(\ider)$, on the de Rham cohomology group $\emph{H}^*(X, \Q)$.  If the kernel of $\rho_X$ is minimal in the sense that it is equal to $2\Z \times \emph{Aut}^0(X) \ltimes \emph{Pic}^0(X)$, then the number of projective Fourier-Mukai partners of $X$ is bounded by the number of conjugacy classes of maximal unipotent subgroups of the image of $\rho_X$.  In particular, if the image of $\rho_X$ is an arithmetic subgroup of a semisimple lie group, then the number of projective Fourier-Mukai partners of $X$ is finite i.e. Conjecture~\ref{conj:1} holds for $X$ (with regards to projective Fourier-Mukai partners).
\end{theorem}

\begin{proof}
First observe that given any equivalence of categories, $F: \der \to \dery$, the conjugation $F_*$ induces an isomorphism of exact sequences.
\[
\begin{CD}
0      @>>> \text{ker }\rho_X    @>>>   \text{Aut}(\der) @>\rho_X>> \text{im }\rho_X  @>>> 0   \\
@.        @VVV                      @VF_*VV                      @VVV                  @.    \\
0      @>>> \text{ker }\rho_Y    @>>>   \text{Aut}(\dery) @>\rho_Y>> \text{im }\rho_Y  @>>> 0
\end{CD}
\]
In particular (by abuse of notation) we have an isomorphism $F_* : \text{ker }\rho_X \tilde{\longrightarrow} \text{ker }\rho_Y$.  A theorem of R. Rouquier (see \cite{H,Ro}) and independently due to Rosay (see \cite{Ros}), states that $F_*$ induces an isomorphism of group schemes $F_*: \text{Aut}^0(X) \ltimes \text{Pic}^0(X) \tilde{\longrightarrow} \text{Aut}^0(Y) \ltimes \text{Pic}^0(Y)$.  In fact, in Rosay's thesis, he proves that $\text{Aut}(\der)$ is a group scheme and $\text{Aut}^0(X) \ltimes \text{Pic}^0(X)$ is the connected component of the identity functor.  We are able to conclude that if $Y$ is any Fourier-Mukai partner of $X$, then $\text{ker }\rho_Y = 2\Z \times \text{Aut}^0(Y) \ltimes \text{Pic}^0(Y)$.

Let $Y$ and $Z$ be two projective Fourier-Mukai partners of $X$ with ample line bundles $\A_Y$ and $\A_Z$ respectively.  Fix equivalences $F: \dery \to \der$ and $G: \derz \to \der$.  Notice that the action of an ample line bundle on the cohomology of a space is given by multiplication by the Chern character and thus is unipotent.  Let $y=\rho_X(F_*(\mu_{\A_Y}))$ and $z =\rho_X(G_*(\mu_{ \A_Z}))$.  It follows that $y$ and $z$ are both unipotent.

Now suppose that $y$ and $z$ lie in maximal unipotent subgroups which are conjugate.  Then by altering one of the equivalences by an autoequivalence, as in the proof of Corollary~\ref{cor:maximalabelian}, we may assume they lie in the same maximal unipotent subgroup.
Inductively define $z_i = [y,z_{i-1}] = yz_{i-1}y^{-1}{z_{i-1}}^{-1}$ with $z_0 = z$.  The lower central series of any unipotent group terminates, in particular, there exists an $n$ such that $z_n = \text{id}$.  For each $i \neq 0$ choose a lift $Z_i$ of $z_i$ to $\text{Aut}(\dery)$ so that one has $\rho_X(F_*(Z_i)) = z_i$.  For $i=0$ we set $Z_0 = (F^{-1} \circ G)_*(\mu_{\mathcal A_Z})$.

We now proceed by backward induction with $n$ the base case and ending at zero.  Assume that $Z_i \cong \mu_{N_i} \circ \gamma_i [s_i]$ for some line bundle $\mathcal N_i \in \text{Pic}(Y)$, some automorphism $\gamma_i \in \text{Aut}(Y)$ and some $s_i \in \Z$.  We want to show that the same holds for $i-1$.  First we look at the base case where $z_n=\text{Id}$.  It follows from the discussion in the first paragraph of the proof, that  $Z_n \in \text{ker }\rho_Y = 2\Z \times \text{Aut}^0(Y) \ltimes \text{Pic}^0(Y)$. Thus, the induction hypothesis is satisfied for $i=n$.

Now, by definition $z_{i-1}yz_{i-1}^{-1} = yz_{i}$.  Lifting this equation to $\dery$, there exists a $\psi \in \text{ker }\rho_Y$ such that ${Z_{i-1}}_*(\mu_{\A_Y}) \cong \mu_{\A_Y} \circ Z_i  \circ \psi$.  Again by the discussion in the first paragraph of the proof, $\psi \cong \mu_{\L} \circ \alpha_* [2t]$ for some line bundle $\L \in \text{Pic}^0(Y)$, some automorphism $\alpha \in \text{Aut}^0(Y)$, and some $t \in \Z$.
We obtain the following sequence of isomorphisms of functors,
\begin{align*}
{Z_{i-1}}_*(\mu_{\A_Y}) &\cong \mu_{\A_Y} \circ Z_i  \circ \psi\\
& \cong \mu_{\A_Y} \circ \mu_{\mathcal N_{i-1}} \circ {\gamma_{i-1}}_* \circ \psi [r_{i-1}]\text{ (by the induction hypothesis)} \\
& \cong \mu_{\A_Y} \circ \mu_{\mathcal N_{i-1}} \circ {\gamma_{i-1}}_* \circ \mu_{\L} \circ \alpha_*[2t] [r_{i-1}]\text{ (as }\psi \in \text{ker }\rho_Y) \\
& \cong \mu_{\A_Y \otimes \mathcal N_{i-1} \otimes {\gamma_{i-1}}_*\mathcal L} \circ (\gamma_{i-1}  \circ \alpha)_* [2t + r_{i-1}].
\end{align*}
Applying Lemma~\ref{lem:twist}, one obtains the induction step.  We conclude that $Z_0 \cong \mu_{\mathcal N} \circ \gamma_* [r]$ for some $\mathcal N \in \text{Pic}(Y)$ ,  $\gamma \in \text{Aut}(Y)$, and $r \in \Z$.  

Finally, $(F^{-1} \circ G)_* (\mu_{\mathcal A_Z}) = Z_0 \cong \mu_{\mathcal N} \circ \gamma_*[r]$.  Applying Lemma~\ref{lem:twist} once again, we obtain an isomorphism between $Y$ and $Z$.  This tells us that the number of Fourier-Mukai partners of $\der$ is bounded by the number of conjugacy classes of maximal unipotent subgroups of the image of $\rho_X$.  The fact that the number of conjugacy classes of maximal unipotent subgroups of an arithmetic group is finite is well-known\footnote[1]{Stated this way the result can be found as \cite{M} Corollary 9.38.   It is equivalent to Theorem 9.37 of \cite{M} which says that for any parabolic subgroup $P$, and arithmetic group $\Gamma$ of an algebraic group $G$ the double-coset space $\Gamma \backslash G_{\Q} / P_{\Q}$ is finite.  The latter statement can be found for example in \cite{B} Theorem 15.6 or \cite{Ra} Theorem 13.26.}.
\end{proof}

\begin{remark} On an even dimensional variety the square of a spherical twist acts trivially on cohomology \cite{ST} and similarly any $\P^n$-twist acts trivially on cohomology \cite{HT}.  Hence any variety whose derived category possesses  such an autoequivalence, does not satisfy the hypotheses of the above theorem.
\end{remark}

We now apply our theorem to the case of abelian varieties.  The autoequivalences of the derived category of an abelian variety have been well studied (see \cite{GLO,H, O, P}).  Indeed, the conditions of Theorem ~\ref{thm:representation} are satisfied for abelian varieties.  Moreover, using the technology developed in \cite{GLO}, we are able to give an explicit bound in the case where the Neron-Severi group of the abelian variety is isomorphic to $\Z$ (this is a generic condition).

First, let us explain how do obtain a certain numerical invariant of an abelian variety with Neron-Severi group isomorphic to $\Z$.  Suppose $A$ is an abelian variety with Neron-Severi group isomorphic to $\Z$.  Let $\mathcal L$ be a representative of a generator of the Neron-Severi group of $A$ and  $\mathcal M$ be a representative of a generator of the Neron-Severi group of $\hat{A}$.  As ample line bundles, $\mathcal L$ and $\mathcal M$ induce isogenies $\alpha_{\mathcal L}$ and $\alpha_{\mathcal M}$.  Taking their composition, one obtains a number, $\alpha_{\mathcal L} \circ \alpha_{\mathcal M} = N \cdot \text {Id}$.

\begin{theorem}
Conjecture~\ref{conj:1} holds (with regards to projective Fourier-Mukai partners) for abelian varieties over $\C$.  Suppose that $A$ is an abelian variety over $\C$ such that the Neron-Severi group of $A$, is isomorphic to $\Z$. Then the number of smooth projective Fourier-Mukai partners of $A$ is bounded by $\sum_{d|N} \phi(\emph{gcd}(d,\frac{N}{d}))$, which is the number of inequivalent cusps of the action of $\Gamma_0(N)$ on the upper half plane (here $\phi$ denotes Euler's Phi-function).  
\end{theorem}

\begin{proof}
From \cite{GLO}, the group $\text{Spin}(A)$ is an arithmetic subgroup of a semisimple lie group hence the conditions of Theorem~\ref{thm:representation} are satisfied.

We recall the definition of the Polishchuk group,
$$U(A):=  \{M \in \text{Aut}(A \times \hat{A}) |\text{ } M^{-1} = \text{det}(M) \overline {{M^{ - 1}}}\}. $$

For an abelian variety we have the following diagram where $\text{Spin}(A) = \text{im }\rho$ \cite{GLO,H},
$$\begin{CD}
@.@.@.@.@.\\
@.            @.                                @.                               0                  @.           \\
@.            @.                                @.                               @VVV                   @.           \\
@.            @.                                @.                               \Z/2\Z                   @.           \\
@.            @.                                 @.                               @VVV                     @.          \\
0      @>>>   2\Z \times A \times \hat{A}    @>>>   \text{Aut}(D^b(A)) @>\rho>>   \text{Spin}(A)            @>>> 0           \\
@.            @VVV                                 @VVV                          @V \pi VV                   @.          \\
0      @>>>   \Z \times A \times \hat{A}   @>>>   \text{Aut}(D^b(A)) @>>>  \Gamma_0(N)  @>>> 0              \\
@.            @VVV                                 @.                               @.                     @.          \\
@.           \Z/2\Z                                 @.                               @.                   @.        \\
@.            @VVV                              @.                               @.                     @.          \\   
@.            0                                @.                               @.                   @.        \\
@.@.@.@.@.\\
\end{CD}$$
\vspace{ .1 in}
Also, suppose $x$ and $y$ be unipotent elements of $\text{Spin}(A)$.  Assume $\pi(x)=\pi(y)$.  From the exact sequence $x = \pm y$, however since $x$ and $y$ are unipotent,  $x$ and $y$ must in fact be equal.  It follows that the maximal unipotent subgroups of $\text{Spin}(A)$ are in bijection with the maximal unipotent subgroups of $U(A)$. Hence we have reduced the study of maximal unipotent subgroups of $\text{Spin}(A)$ to maximal unipotent subgroups of $U(A)$.

Now suppose that $A$ is an abelian variety over $\C$ such that the Neron-Severi group of $A$, $\text{NS}(A)$, is isomorphic to $\Z$.  We would like to calculate the number of maximal unipotent subgroups of $\text{U}(A)$.  Let $M \in U(A)$ be a unipotent element.  As,  $M^{-1} = \text{det}(M) \overline {{M^{ - 1}}}$, one concludes that $M$ is real.  Furthermore, as $NS(A) = \Z$,  if $M$ is real, then $M$ must have integer entries.  Playing around a bit, one sees that the integral part of $U(A)$ is isomorphic as an arithmetic group to $\Gamma_0(N)$ (see also \cite{O}).  Hence any maximal unipotent subgroup of $U(A)$ is also a maximal unipotent subgroup of $\Gamma_0(N)$.  Thus we have reduced our problem further to the study of maximal unipotent subgroups of $\Gamma_0(N)$.  

The number of maximal unipotent subgroups of $\Gamma_0(N)$ is a classical problem and we briefly remind the reader of how it goes.  For a more complete picture, see for example \cite{DS}.  A matrix in $\Gamma_0(N)$ is unipotent if and only if the trace is two which occurs if and only if the action of the matrix on the upper half plane is parabolic.  Such a matrix fixes a unique cusp on the boundary of the upper half plane.  We say that two such cusps are equivalent $z_1 \sim z_2$ if and only if there exists $\gamma \in \Gamma_0(N)$ such that $\gamma(z_1) =z_2$.    Now suppose $\gamma(z_1) = z_2$ and $B$ is a unipotent matrix (hence parabolic) which uniquely fixes the cusp $z_1$. Then $\gamma^{-1}B\gamma$ uniquely fixes $z_2$.  Furthermore, suppose that $B$ is not a power of any other unipotent matrix.  One easily verifies that the subgroup generated by $B$ is a maximal unipotent subgroup.  As every cusp is fixed by some unipotent matrix in $\Gamma_0(N)$, it follows that the number of inequivalent cusps is precisely the number of maximal unipotent subgroups.  This is a well studied phenomenon.  The number of such classes of cusps is precisely, $\sum \limits_{d|N} \phi(\text{gcd}(d,\frac{N}{d}))$ (\cite{DS}, pg. 103).

\end{proof}

\noindent \textbf{Acknowledgments:} I would like to thank my adviser, Tony Pantev, for guiding me toward these ideas, endless patience, and extremely stimulating discussions.  Without his help, I would not have even known where to begin.  I would also like to thank Dmitri Orlov for several enlightening conversations and ideas, David Witt Morris for pointing out several references in the theory of arithmetic groups, and David Fithian for a useful discussion on modular forms.  This work was funded by NSF Research Training Group Grant, DMS 0636606.

\end{document}